\documentclass[preprint,12pt]{elsarticle}
\usepackage{amssymb}
\usepackage[fleqn]{amsmath}
\usepackage{amsthm}
\usepackage{tikz}
\usepackage{epstopdf}
\usepackage{lineno,hyperref}
\newtheorem{thm}{Theorem}

\newtheorem{lem}[thm]{Lemma}
\newtheorem{rem}{Remark}
\journal{https://arxiv.org}

\begin{document}

\begin{frontmatter}

\title{\Large $L^p$-$L^q$ boundedness of integral operators with oscillatory kernels: Linear versus quadratic phases}

\author{Ahmed A. Abdelhakim}
\begin{abstract}
Let $\,T^{j,k}_{N}:L^{p}(B)\,
\rightarrow\,L^{q}([0,1])\,$
be the oscillatory integral operators defined by
$\;\displaystyle T^{j,k}_{N}f(s):=\int_{B}
\,f(x)\,e^{\imath N{|x|}^{j}s^{k}}\,dx,
\quad (j,k)\in\{1,2\}^{2},\,$
where $\,B\,$ is the unit ball in ${\mathbb{R}}^{n}\,$
and $\,N\,>>1.$ We compare the asymptotic behaviour as $\,N\rightarrow +\infty\,$ of the operator norms
 $\,\parallel T^{j,k}_{N} \parallel_
{L^{p}(B)\rightarrow L^{q}([0,1])}\,$
for all
$\,p,\,q\in [1,+\infty].\,$
We prove that, except for the dimension $n=1,\,$
this asymptotic behaviour depends on the linearity or quadraticity of the phase in $s$ only.
We are led to this problem by an observation on
inhomogeneous Strichartz estimates for the Schr\"{o}dinger equation.
\end{abstract}

\begin{keyword}
Strichartz estimates for the Schr\"{o}dinger equation \sep
Oscillatory integrals\sep $L^{p}-L^{q}$ boundedness
\MSC[2010] 35B45, 35Q55, 42B20
\end{keyword}

\end{frontmatter}

\subsubsection*{1. A remark on a counterexample
to inhomogeneous Strichartz estimates for
the Schr\"{o}dinger equation and motivation}
Consider the Cauchy problem
for the inhomogeneous free Schr\"{o}dinger  equation with zero initial data
\begin{align}\label{shreq}
\imath \partial_{t}u+\Delta u\,=\, F(t,x),\qquad (t,x)\in (0,\infty)\times{\mathbb{R}}^{n},\qquad u(0,x)\,=\,0.
\end{align}
Space time estimates of the form
\begin{align}\label{est1}
||u||_{L^{q}_{t}\left(
\mathbb{R};L^{r}_{x}({\mathbb{R}}^{n})\right)}\;\lesssim\;
||F||_{L^{{\widetilde{q}}^{\prime}}_{t}
\left(\mathbb{R};L^{{\widetilde{r}}
^{\prime}}_{x}({\mathbb{R}}^{n})\right)},
\end{align}
have been known as inhomogeneous Strichartz estimates.
The results obtained so far (see \cite{damianoinhom,
Kato,keeltao,vilela,YoungwooKoh})
are not conclusive when it comes to determining  the optimal values of the Lebesue exponents $\,q$, $r$, $\tilde{q}\,$ and $\,\tilde{r}\,$ for which
the estimate (\ref{est1}) holds. Trying to further understand this problem, we \cite{ahmed1} found
new necessary conditions on these exponents values.
The counterexample in \cite{ahmed1}, like
Example 6.10 in \cite{damianoinhom},
contains an oscillatory factor with high frequency. More precisely, we used
a forcing term given by
\begin{align}\label{myphase}
 F(t,x)= e^{-\imath\, N^2\,t } \,\chi_{[0,\frac{\eta}{N}]}(t)\,
 \chi_{B\left(\frac{\eta}{N}\right)}{(x)}
\end{align}
where $\,\eta>0\,$ is a fixed small number,
$\,N>>1\,$ and $ B\left(\frac{\eta}{N}\right)$ is the ball with radius $\,\eta/N\,$ about the origin.
While in \cite{damianoinhom}
the stationary phase method is applied to
the inhomogeneity
\begin{align}\label{fphase}
 F (t,x)= e^{-2\imath\, N^2\,t^{2} }\,\chi_{[0,1]}(t)\,\chi_{B\left(\frac{\eta}{N}\right)}{(x)}.
\end{align}
When $\,t\in [2,3],\,$
both data in (\ref{myphase}) and (\ref{fphase})
force the corresponding solution $u(t,x)$ to concentrate in a spherical shell centered at the origin with radius about $N.$ This agrees with the dispersive nature
of the Schr\"{o}dinger operator.
The shell thickness is different in both cases
though. It is about $1$ in the case of the data
(\ref{myphase}) but about $N$ in the case
of (\ref{fphase}). The necessary conditions obtained
are respectively
\begin{align*}
\frac{1}{q}\geq\frac{n-1}{\widetilde{r}}-\frac{n}{r},
\qquad  \quad\frac{1}{\widetilde{q}}\geq\frac{n-1}{r}-
\frac{n}{\widetilde{r}}
\end{align*}
and
\begin{align}\label{necess1}
|\frac{1}{r}-\frac{1}{\widetilde{r}}|\leq
\frac{1}{n}.
\end{align}
Observe that the oscillatory function in (\ref{myphase})
has a linear phase and is applied for the short time period of length $\:1/\sqrt{\text{frequency}}.\,$
The oscillatory function in (\ref{fphase})
on the other hand has a quadratic phase and
the oscillation is put to work for a whole time unit.
We noticed that the phase in \cite{damianoinhom}
need not be quadratic and we can get
the necessary condition (\ref{necess1}) using the data
\begin{align}\label{fphase1}
 F_{l} (t,x)= e^{-\imath\, N^2\,t}\,
 \chi_{[0,1]}(t)\,\chi_{B\left(\frac{\eta}{N}\right)}{(x)}
\end{align}
where the phase in the oscillatory function is linear.
Before we show this, we recall the following
approximation of oscillatory integrals according to
the principle of stationary phase.
\begin{lem}\label{stationary}
(see \cite{stein}, Proposition 2 Chapter VIII
and Lemma 5.6 in \cite{damianorem})
Consider the oscillatory integral $\;I(\lambda)=\displaystyle \int_{a}^{b}e^{\imath \lambda \phi(s)}\psi(s)d s.\;$ Let the phase $\,\phi \in C^5([a,b])\,$ and the amplitude $\psi\in C^3([a,b])$
such that
\begin{description}
\vspace*{-0.22 cm}
\item[(i)] \hspace{0.1 cm} $\;\phi^{\prime} (z)=0\,$ for a point $\;z\,\in\, ]a+c, b-c[\;$ with $\,c\,$ a positive constant,
\vspace*{-0.22 cm}
\item[(ii)] \hspace{0.0001 cm} $\;|\phi^{\prime} (s)|\,\gtrsim\, 1,\;$
for all $\;s\,\in\, [a, a+c]\,\cup\, [b-c, b],$
\vspace*{-0.22 cm}
\item[(iii)] $\;|\phi^{\prime\prime} (s)|\,\gtrsim\, 1,$
\vspace*{-0.22 cm}
\item[(iv)] $\;\psi^{(j)}\,$ and $\,\phi^{(j+3)}\;$
are uniformly bounded on $[a,b]$ for all $j=0,1,2$.\vspace*{-0.22 cm}
\end{description}
 \begin{align*}
\hspace*{-1 cm}\mbox{Then}\qquad \qquad  I(\lambda)\,=\, {\,\sqrt{\frac{2 \pi}{\lambda|\phi^{\prime\prime} (z)|}} \,\psi(z)\,e^{\imath \,\lambda\, \phi(z)+\imath\,\mbox{\small sgn}\left(
\phi^{\prime\prime} (z)\right)\, \frac{\pi}{4}}}+\mathcal{O}\left(\frac{1}{\lambda}\right),
\vspace*{-0.22 cm}
\end{align*}
where the implicit constant in the $\mathcal{O}-$symbol is absolute.
\end{lem}
The norm of the inhomogeneous term $F_{l} $
in (\ref{fphase1}) has the estimate
\begin{align}\label{normf}
\parallel F_{l} \parallel_{L^{\tilde{q}^{\prime}}
([0,1];L^{\tilde{r}^{\prime}}(\mathbb{R}^{n}))}\,\approx\,
{\eta}^{n-\frac{n}{\tilde{r}}}\,
{N}^{-n}\,{N}^{\frac{n}{\tilde{r}}}.
\end{align}
For the solution of (\ref{shreq}), we have the explicit formula
\begin{align}\label{solshro}
 u(t,x) \,=\, (4 \pi)^{-\frac{n}{2}}\int_{0}^{t}
(t-s)^{-\frac{n}{2}}\int_{{\mathbb{R}}^{n}}
e^{\imath\frac{|x-y|^2}{4(t-s)}}\,F(s,y)\,dy\, ds.
\end{align}
Let us estimate the solution
$u_{l}(t,x)$ that corresponds to $F_{l}.$
We shall restrict our attention to the region
\begin{align*}
\Omega_{\eta,N}=\left\{
(t,x)\in [2,3] \times \mathbb{R}^{n}\!\!:\,
2(t-{3}/{4})N+\eta N^{-1}\,<|x|<\,2(t-{1}/{4})N-\eta N^{-1}\right\}.
\end{align*}
It will be momentarily  seen
that this is the region
where we can exploit Lemma \ref{stationary} to approximate $\, u_{l}(t,x).$
Substituting from (\ref{fphase1}) into (\ref{solshro})
then applying Fubini's theorem we get
\begin{align}\label{corsol}
 u_{l}(t,x)\:=\:
(4\pi)^{-\frac{n}{2}} \,\int_{B({\eta}/{N})}
  \,I_{N}(t,x,y)\, d y
\end{align}
where  $\,I_{N}(t,x,y)\,$ is the oscillatory integral
\begin{align}\label{inoscints}
I_{N}(t,x,y)\;=\;\int_{0}^{1}\,
 e^{\imath N^2\, \phi_{N}{(s,t,x,y)}}\,\psi{(s,t)}\, d s,
\end{align}
with the phase $\;
\displaystyle \phi_{N}{(s,t,x,y)} =\frac{ |x-y|^2}{4\,N^2}\frac{1}{t-s}-s\,$
 and amplitude $\, \psi{(s,t)}=(t-s)^{-\frac{n}{2}}.$
\vspace{0.1 cm} For simplicity, we write $\phi(.)$ and $\psi(.)$ in place of $\,\phi_{N}{(.,t,x,y)}\,$ and $\,\psi{(.,t)}\,$ respectively.
Next, we verify the conditions
(\textbf{i}) - (\textbf{iv}) for $\,\phi\,$ and $\,\psi.$
Let $\,(t,x)\in \Omega_{\eta,N}\,$ and
$\,y\in B(\eta/N).$ Observe then that
$\;\displaystyle t-{3}/{4}<{|x-y|}/{2N}<t-{1}/{4}\;$
and $\; t-s \in [1,3]. $ Therefore
\begin{description}
\item[(i)] If $\,z\,$ is such that $\,\phi^{\prime}(z)=0\,$
then $\,\displaystyle
 z =t-{|x-y|}/{2N}.\,$ Moreover,
$\displaystyle \,z\in\:]{1}/{4},{3}/{4} [.$
\vspace*{-0.15 cm}
\item[(ii)] $\phi^{\prime}$ is monotonically increaing
so $\;\displaystyle \min_{s\in[0,1]}{\phi^{\prime}(s)}
= \phi^{\prime}(0)=\frac{|x-y|^2}{4 N^2\, t^2}
> \left(1-\frac{3}{4t}\right)^{2}
\gtrsim 1.$
\vspace*{-0.25 cm}
\item[(iii)]$\,\displaystyle
\phi^{\prime\prime}(s)\,=\,
\frac{|x-y|^2}{2 N^2}\frac{1 }{(t- s)^3}\,\approx\,1.$
\vspace*{-0.2 cm}
\item[(iv)]  $\;\displaystyle \phi^{(j)}(s)\,=\,
\frac{|x-y|^2}{4N^2}\frac{ j!}{(t- s)^{(j+1)}} \,\approx\,1,\; j=3,4,5$,
$\;\;\;\psi(s)\,=\,
(t- s)^{-\frac{n}{2}}
\,\approx\,1$,\\[0.15 cm]
$\psi^{\prime}(s)\,=\,
\frac{n}{2}(t- s)^{-\frac{n}{2}-1}
\,\approx\,1$, $\qquad \psi^{\prime\prime}(s)\,=\,
\frac{n}{2}(\frac{n}{2}+1)(t- s)^{-\frac{n}{2}-2}\,\approx\,1.$
\end{description}
Now, applying Lemma \ref{stationary} to the oscillatory integral $\,I_{N}(t,x,y)\,$ in (\ref{inoscints}) yields
\begin{align}\label{noosciny0}
I_{N}(t,x,y)\,=\,
{\,\sqrt{\frac{2 \pi}{\phi_N^{\prime\prime}(z,t,x,y)}} \,\psi(z,t)\,
\frac{e^{ \frac{\pi}{4} \imath}}{N}
\,e^{\imath N^2 \phi_N(z,t,x,y)}}+
\mathcal{O}\left(\frac{1}{N^2}\right).
\end{align}
Since $\,\phi_N(z,t,x,y)+t=|x-y|/N\, $ and
since $\, N\left(|x-y|-|x|\right)= \mathcal{O}\left(\eta\right)\,$
whenever \\ $\,(t,x)\in \Omega_{\eta,N},\;y\in B(\eta/N).$
Then
$\,\,N^2\,\phi_N(z,t,x,y)+N^2\,t
=N\,|x|+\mathcal{O}\left(\eta\right).$ Hence
\begin{align}\label{noosciny}
e^{\imath N^2\,\phi_N(z,t,x,y)}=
e^{\imath\left(N\,|x|-N^2\,t\right)}
\,e^{\mathcal{O}\left(\eta\right)}
=e^{\imath\left(N\,|x|-N^2\,t\right)}\,
\left(1+\mathcal{O}\left(\eta\right)\right).
\end{align}
Inserting (\ref{noosciny}) into (\ref{noosciny0})
then returning to (\ref{corsol}), we discover
\begin{align*}
 u_{l}(t,x)\:=\:&
\frac{(4\pi)^{\frac{1-n}{2}} }{\sqrt{2}}\frac{e^{ \frac{\pi}{4} \imath}}{N} \,e^{\imath\left(N\,|x|-N^2\,t\right)}\,
\int_{B({\eta}/{N})}\,{
\,\frac{\psi(z,t)}{\sqrt{\phi_N^{\prime\prime}(z,
t,x,y)}} \,
\,\left(1+\mathcal{O}\left(\eta\right)\right)}
\,d y\\&\;+
\mathcal{O}\left(\frac{1}{N^2}\right)\,
\int_{B({\eta}/{N})}\,\,
d y.
\end{align*}
Recalling that $\,\psi,\:
\phi^{\prime\prime}\approx 1,\,$ we immediately deduce
the estimate
\begin{align*}
&| u_{l}(t,x)|\,\gtrsim\,\frac{|B(\eta/N)|}{N}
\,\approx\,\eta^{n}\,N^{-(1+n)},\quad
(t,x)\in \Omega_{\eta,N}.\quad \text{Thus, for all}\;\;
t\in [2,3],\\
&\hspace*{-1 cm}
||u_{l}(t,x)||_{L^{r}_{x}\left({\mathbb{R}}^{n}\right)}
\,\geq\,
\left( \int_{2(t-{3}/{4})N+\eta N^{-1}\,<\,|x|\,<\,2(t-{1}/{4})N-\eta N^{-1}}\,
| u_{l}(t,x)|^{r}\,dx\,\right)^{\frac{1}{r}}
\,\gtrsim\,\eta^{n}\,N^{-(1+n)+\frac{n}{r}}.
\end{align*}
Consequently
\begin{align}\label{normul}
||u_{l}||_{L^{q}_{t}\left(
\mathbb{R};L^{r}_{x}({\mathbb{R}}^{n})\right)}
\,\geq\,||u_{l}||_{L^{q}_{t}\left(
[2,3];L^{r}_{x}({\mathbb{R}}^{n})\right)}
\,\gtrsim\,\eta^{n}\,N^{-(1+n)+\frac{n}{r}}.
\end{align}
Lastly, it follows from (\ref{normf}) and
(\ref{normul}) that
\begin{align*}
||u_{l}||_{L^{q}_{t}\left(
\mathbb{R};L^{r}_{x}({\mathbb{R}}^{n})\right)}/
\parallel F_{l} \parallel_{L^{\tilde{q}^{\prime}}
([0,1];L^{\tilde{r}^{\prime}}(\mathbb{R}^{n}))}\;
\gtrsim\;
\eta^{\frac{n}{\tilde{r}}}\,
N^{\frac{n}{r}-\frac{n}{\tilde{r}}-1}
\end{align*}
which, for a fixed $\,\eta,\,$ blows up
as $\,N\rightarrow +\infty\,$
if $\,\displaystyle \frac{n}{r}-\frac{n}{\tilde{r}}>1.\,$
In the light of duality this implies
the necessary condition (\ref{necess1}).

These examples made us wonder
how exactly different are linear oscillations from
quadratic ones if we capture the cancellations in
Lebesgue spaces.
One way to see this is to
consider the operators
$\,T^{j,k}_{N}:L^{p}(B)\,
\rightarrow\,L^{q}([0,1])\,$ defined by
\begin{align}\label{intop}
T^{j,k}_{N}f(s):=\int_{B}
\,f(x)\,e^{\imath N{|x|}^{j}s^{k}}\,dx,
\qquad (j,k)\in\{1,2\}^{2},
\end{align}
where $\,B\,$ is the unit ball in ${\mathbb{R}}^{n},\,$
and compare the asymptotic behaviour as $\,N\rightarrow +\infty\,$ of their operator norms for all
$\,p,\,q\in [1,+\infty].\,$
Let $\,C_{j,k,n}:[0,1]^{2}\rightarrow \mathbb{R}\,$
be the functions defined by
\begin{align*}
C_{j,k,n}\left(\frac{1}{p},\frac{1}{q}\right)\,:=\,
\alpha \quad \text{if}\qquad
\parallel T^{j,k}_{N}\parallel_{L^{p}\left(B\right)
\rightarrow L^{q}([0,1])}
\;\approx\; N^{ - \alpha}.
\end{align*}
We discover that $\,C_{j,k,n}\,$ is a continuous function
with range $\,[0,{1}/{4}]\,$
when $n=1,$ $j=2$ and  $\,[0,{1}/{2}{k}]\,$
 otherwise (see the figure below).
We actually prove that
\begin{thm}\label{mainthm}
\begin{equation*} C_{j,k,n}\left(\frac{1}{p},\frac{1}{q}\right)\;=\;
\left\{
  \begin{array}{ll}
\frac{1}{4}\, \sigma\left(\frac{1}{p},\frac{1}{q}\right), & \hbox{$n=1,\;$ $j=2$;} \\\\
\frac{1}{2\,k}\, \sigma\left(\frac{1}{p},\frac{1}{q}\right), & \hbox{
$n\geq j$.}
  \end{array}
\right.
\end{equation*}
where
\begin{equation}\label{sgmab}
 \sigma(a,b):=\left\{
                  \begin{array}{ll}
                   2b , & \hbox{$\; 0\leq a \leq 1-b,\;\;
0\leq b \leq \frac{1}{2}$;} \\
                   2(1-a) , & \hbox{$\; \frac{1}{2}\leq a \leq 1,\;\;a+ b \geq 1$;} \\
                   1 , & \hbox{$\;0\leq a \leq \frac{1}{2},\;\;\frac{1}{2}\leq b \leq 1$.}
                  \end{array}
                \right.
\end{equation}
 \end{thm}
\begin{align*}
 \begin{tikzpicture} [scale=9]
     \draw[->] (0.0, 0) -- (0.6, 0) node[below] {$\frac{1}{p}$};
    \draw[->] (0,0.0) -- (0, 0.6)  node[left] {$\frac{1}{q}$};
    \draw (0.5, 0) node[below] {${1}$};
    \draw (0, 0.5) node[left] {${1}$};
        \draw (0.25, 0) node[below] {$\frac{1}{2}$};
    \draw (0, 0.25) node[left] {$\frac{1}{2}$};
    \draw
    (0, 0) -- (0.5, 0.0) -- (0.5, 0.5) -- (0.0, 0.5) -- cycle;
    \draw
    (0.5, 0.0) -- (0.25, 0.25)-- (0.25, 0.5);
   \draw (0.25, 0.25) -- (0.0, 0.25);
\draw [loosely dotted] (0.25, 0.0) -- (0.25, 0.25);
\draw (0.165,0.125) node{$\;\frac{1}{2}\frac{1}{q}$};
\draw (0.125,0.375) node{$\;\frac{1}{4}$};
\draw (0.375,0.3) node{$\;
\frac{1}{2}(1-\frac{1}{p})$};
\draw (0.3, -0.1) node[below] {$C_{2,k,1}$};
\end{tikzpicture}\qquad\qquad\quad
 \begin{tikzpicture} [scale=9]
     \draw[->] (0.0, 0) -- (0.6, 0) node[below] {$\frac{1}{p}$};
    \draw[->] (0,0.0) -- (0, 0.6)  node[left] {$\frac{1}{q}$};
    \draw (0.5, 0) node[below] {${1}$};
    \draw (0, 0.5) node[left] {${1}$};
        \draw (0.25, 0) node[below] {$\frac{1}{2}$};
    \draw (0, 0.25) node[left] {$\frac{1}{2}$};
    \draw
    (0, 0) -- (0.5, 0.0) -- (0.5, 0.5) -- (0.0, 0.5) -- cycle;
    \draw
    (0.5, 0.0) -- (0.25, 0.25)-- (0.25, 0.5);
   \draw (0.25, 0.25) -- (0.0, 0.25);
\draw [loosely dotted] (0.25, 0.0) -- (0.25, 0.25);
\draw (0.165,0.125) node{$\;\frac{1}{k}\frac{1}{q}$};
\draw (0.125,0.375) node{$\;\frac{1}{2}\frac{1}{k}$};
\draw (0.375,0.3) node{$\;
\frac{1}{k}(1-\frac{1}{p})$};
\draw (0.3, -0.1) node[below] {$C_{j,k,n}$};
\end{tikzpicture}
\end{align*}
\begin{rem}
For each $\,p,q\in [1,\infty],$
and all dimension $\,n>1,\,$
the asymptotic behaviour of
$\,\parallel T^{j,k}_{N}\parallel_{L^{p}\left(B\right)
\rightarrow L^{q}([0,1])}\,$
as $\,n\rightarrow +\infty\,$
is determined only by the linearity or quadraticity of the phase in $s$. The role
of the power $j$ of $x$
appears exclusively in the dimension $n=1.$
\end{rem}
\begin{rem}
There is nothing special about
neither the unit interval nor the unit ball
in defining the operators $T^{j,k}_{N}$.
Actually we shall make use of H\"{o}lder
inclusions of $L^{p}$ spaces
on measurable sets of finite measure
(see Lemma \ref{holder} below).
So we may take any suitable two such
sets provided their finite measures are
asymptotically equivalent to a constant
independent of $ N $ as $ N\rightarrow +\infty.$
\end{rem}
Foschi \cite{damianorem} studied
a discrete version of an operator a little simpler than the integral operator $\,T^{1,1}_{N}.\,$ He considered the operator $\,D_{N}:\ell^{p}(\mathbb{C}^{N})
\rightarrow L^{q}(-\pi,\pi)\,$
that assigns to each vector
$\,a=(a_{0},a_{1},...a_{N-1})\in {\mathbb{C}}^{N}\,$ the $\,2\pi$-periodic trigonometric polynomial $\,D_{N}a(t)=\sum_{m=0}^{N-1}a_{m}\,e^{\imath\,m\,t}\,$
and described the asymptotic behaviour
of $\, \displaystyle\sup_{a\in \mathbb{C}^{N}-\{0\}}
{{\parallel D_{N}a\parallel_{L^{q}([-\pi,\pi])} }/
{\parallel a\parallel_{\ell^{p}\left(\mathbb{C}^{N} \right)}}}$ as $N\rightarrow+\infty,$
 for all $ 1\leq p,\,q\leq+\infty.$
The norms there are defined by
\begin{align*}
&\parallel a \parallel_{\ell^{p}}=
\left( \sum_{m=0}^{N-1}|a_{m}|^{p}\right)^{\frac{1}{p}},
\quad
1\leq p <\infty, \qquad
\parallel a \parallel_{\ell^{\infty}}=
\max_{0\leq m\leq N-1}|a_{m}|,\\
&
 \parallel f \parallel_{L^{q}}=
\left(\frac{1}{2\pi} \int_{-\pi}^{\pi}|f(t)|^{q}dt\right)^{\frac{1}{q}},
\quad
1\leq q <\infty, \qquad
\parallel f \parallel_{L^{\infty}}=
\max_{|t|\leq \pi}|f(t)|.
\end{align*}
This was followed by
a similar investigation (see Section 5 in \cite{damianorem}) of a linear
integral operator with an oscillatory kernel
 $\, L_{N}: L^{p}([0,1])\rightarrow L^{q}([0,1])\,$
defined by
\begin{align*}
L_{N}f(t)\,:=\,\int_{0}^{1}\,
e^{\imath N/(1+t+s)}\,\frac{f(s)}{(1+t+s)^{\gamma}}\,ds,
\quad\text{\small for some fixed}\;\;  \gamma \geq 0.
\end{align*}
\subsubsection*{2.  Proof of Theorem \ref{mainthm}}
In order to show Theorem \ref{mainthm}, we
shall go through the following steps.\\
\textbf{\emph{\underline{Step 1}}}.\hspace{0.1 cm}
Find lower bounds
for $\,\parallel T^{j,k}_{N} \parallel_
{L^{p}(B)\rightarrow L^{q}([0,1])}\,$
for all $\,p,q\in [1,+\infty]\,$:\\
Test the ratio $ \parallel T^{j,k}_{N}f \parallel_{L^{q}([0,1])}/
\parallel f \parallel_{L^{p}(B)} $ for
functions $f \in {L^{p}(B)}$
that kill or at least slow down the oscillations in the integrals $T^{j,k}_{N}f.\,$ Of course this ratio is majorized by $\displaystyle  \parallel T^{j,k}_{N} \parallel_
{L^{p}(B)\rightarrow L^{q}([0,1])}=\sup_{f\in L^{p}(B)-\{0\}}
{{\parallel T^{j,k}_{N}f\parallel_{L^{q}([0,1])} }/{\parallel f\parallel_{L^{p}\left(B\right)}}}. $
But what is really interesting is the fact that such functions likely maximize the ratio as well.\\
\textbf{\emph{\underline{Step 2}}}.\hspace{0.1 cm}
We find upper bounds
for $\,\parallel T^{j,k}_{N} \parallel_
{L^{p}(B)\rightarrow L^{q}([0,1])}\,$
for all $\,p,q\in [1,+\infty].$
Thanks to interpolation and  H\"{o}lder's inequality,
we merely need an upper bound
for $\parallel T^{j,k}_{N} \parallel_
{L^{2}(B)\rightarrow L^{2}([0,1]).}$
\begin{lem}\label{holder}
Let $\,T^{j,k}_{N}:L^{p}(B)\,
\rightarrow\,L^{q}([0,1])\,$ be as in (\ref{intop}). Assume that
\begin{align}\label{en11}
 \parallel T^{j,k}_{N}f \parallel_{L^{2}([0,1])}
\,\leq\, c_{j,k,N}
\parallel f \parallel_{L^{2}(B)}.
\end{align}
Then
\begin{align}\label{consigma}
 \parallel T^{j,k}_{N} \parallel_{L^{p}(B)
\rightarrow L^{q}([0,1])}
\;\lesssim_{p,q,n}\; c^{\sigma\left(\frac{1}{p},\frac{1}{q}
\right)}_{j,k,N}
\end{align}
where $\,\sigma:[0,1]^{2}\rightarrow [0,1]\,$
is the continuous function
in (\ref{sgmab}).
\end{lem}
\begin{proof}
If we take
 absolute values of both sides of (\ref{intop}) we get
the trivial estimate \\
$\;\parallel T^{j,k}_{N}f\parallel_{L^{\infty}([0,1])}
\,\leq\,\parallel f\parallel_{L^{1}\left(B\right)}.$
Interpolating this with (\ref{en11})
using Riesz-Thorin theorem (\cite{loukas})
implies
\begin{align}\label{int1}
 \parallel T^{j,k}_{N}f \parallel_{L^{q}([0,1])}
\,\leq\, c^{2\left(1-\frac{1}{p}\right)}_{j,k,N} \parallel f \parallel_{L^{p}(B)},
\qquad \frac{1}{2}\leq\frac{1}{p}\leq 1,\;\;
\frac{1}{q}=1-\frac{1}{p}.
\end{align}
Since, by H\"{o}lder's inequality,
$\; \parallel T^{j,k}_{N}f \parallel_{L^{\bar{q}}([0,1])}
\,\leq\,\parallel T^{j,k}_{N}f \parallel_{L^{q}([0,1])}\;$ whenever \\$\,1\leq \bar{q}\leq q\leq \infty,$ then
\begin{align}\label{int2}
 \parallel T^{j,k}_{N}f \parallel_{L^{q}([0,1])}
\,\leq\, c^{2\left(1-\frac{1}{p}\right)}_{j,k,N} \parallel f \parallel_{L^{p}(B)},
\qquad \frac{1}{2}\leq\frac{1}{p}\leq 1,\;\;
1-\frac{1}{p}\leq\frac{1}{q}\leq 1.
\end{align}
Applying H\"{o}lder's inequality once more
we find that if $\;1\leq p\leq\bar{p}\leq \infty,\,$ then
\begin{align}
\nonumber&\hspace{-1 cm}\parallel f \parallel_{L^{p}(B)}
\,\leq\,|B|^{\frac{1}{p}-\frac{1}{\bar{p}}}\,
\parallel f \parallel_{L^{\bar{p}}(B)}.
\;\; \text{Therefore by}\;(\ref{int1})\;
\text{we have}\\
\label{int3}
&\hspace{-0.6 cm}\parallel T^{j,k}_{N}f \parallel_{L^{q}([0,1])}
\,\leq\, |B|^{1-\frac{1}{p}-\frac{1}{q}}\,
c^{2/q}_{j,k,N} \parallel f \parallel_{L^{p}(B)},
\quad 0\leq\frac{1}{q}\leq \frac{1}{2},\;\;
0\leq\frac{1}{p}\leq 1-\frac{1}{q}.
\end{align}
Moreover, since we know from (\ref{int2})
that
\begin{align}
\nonumber &\hspace*{-1 cm}\parallel T^{j,k}_{N}f \parallel_{L^{q}([0,1])}
\,\leq\, c_{j,k,N} \parallel f \parallel_{L^{2}(B)},
\quad \frac{1}{2}\leq\frac{1}{q}\leq 1,\quad
\text{then}\\
&\label{int4}
\parallel T^{j,k}_{N}f \parallel_{L^{q}([0,1])}
\,\leq\, |B|^{\frac{1}{2}-\frac{1}{p}}\,
c_{j,k,N} \parallel f \parallel_{L^{p}(B)},
\quad 0\leq\frac{1}{p}\leq \frac{1}{2},\;\;
\frac{1}{2}\leq\frac{1}{q}\leq 1.
\end{align}\vspace{-0.5 cm}
\end{proof}
If the constants in inequalities (\ref{int1}) - (\ref{int4}) were sharp, they would be precisely the values of the corresponding norms
$\,\parallel T^{j,k}_{N}\parallel_{L^{p}\left(B\right)
\rightarrow L^{q}([0,1])}.$
Unfortunately, we are not able to
compute the optimal constant $\,c_{j,k,N}\,$  in the energy estimate (\ref{en11}).
Nevertheless, the constants $\,c^{\sigma\left(\frac{1}{p},\frac{1}{q}
\right)}_{j,k,N}\,$ in (\ref{consigma})
 would be good enough
for our purpose if,
for each $p,q\in [1,+\infty],$ they were asymptotically equivalent, as $ N\rightarrow +\infty$, to the corresponding lower bounds
of $\,\parallel T^{j,k}_{N}\parallel_{L^{p}\left(B\right)
\rightarrow L^{q}([0,1])}\,$ that we compute in \emph{Step 1}.\\
\textbf{\emph{\underline{Step 1}}. }\\
(i) \textbf{Focusing data}\\
When $\,x\in B(\eta /N^{\frac{1}{j}})\,$ we have
$\;\displaystyle e^{\imath N{|x|}^{j}s^{k}}=
e^{\mathcal{O}\left(\eta\right)}=
1+\mathcal{O}\left(\eta\right), \;\;
\text{\small for all}\;s\in [0,1].$
Thus, if we
take $f_{j}$ to be the focusing functions
$\,f_{j}=\displaystyle \chi_{B(\eta /N^{\frac{1}{j}})}\,$ then $\; \displaystyle \parallel f \parallel_{L^{p}(B)}\,=\,
|B(\eta /N^{\frac{1}{j}})|^{\frac{1}{p}}\;$
and
\begin{align*}
T^{j,k}_{N}f_{j}(s)\,=\,\int_{B(\eta /N^{\frac{1}{j}})}
\,e^{\imath N{|x|}^{j}s^{k}}\,dx=
\int_{B(\eta /N^{\frac{1}{j}})}
\,\left(1+\mathcal{O}\left(\eta\right)\right)\,dx
\,\gtrsim \,|B(\eta /N^{\frac{1}{j}})|
\end{align*}
for all $ \;0\leq s\leq 1.\;$
Consequently, since $\eta$ is fixed,
\begin{align}\label{lb1}
\parallel T^{j,k}_{N}\parallel_{L^{p}\left(B\right)
\rightarrow L^{q}([0,1])}
\,\geq\,\frac{\parallel T^{j,k}_{N}f_{j} \parallel_{L^{q}([0,1])}}{
\parallel f_{j} \parallel_{L^{p}(B)}}
\;\gtrsim\;
N^{-\frac{n}{j}\left(1-\frac{1}{p}\right)}.
\end{align}
The figure below illustrates the one dimensional case.
\begin{align*}
&\hspace{-1 cm} \begin{tikzpicture} [scale=10]
\fill[fill = black!10] (0,0.35)--(0.048,0.35)--(0.048,0.0)--
(0,0)--cycle;
     \draw[->] (0, 0) -- (0.38, 0) node[below] {\small$x$};
    \draw[->] (0,0.0) -- (0, 0.4)  node[left] {\small$ f_{j}(x)$};
        \draw (0.35, 0) node[below] {\small ${1}$};
        \draw (0,0.35) node[left] {\small ${1}$};
\draw[thick] (0,0.35)--(0.05,0.35);
\draw[thick] (0.05,0.0001)--(0.35,0.0001);
\draw (0.07,0.0)node[below] { \footnotesize {$N^{-\frac{1}{j}}$}};
 \draw[dotted] (0.35, 0) -- (0.35,0.35)--(0,0.35);
\end{tikzpicture}\quad
  \begin{tikzpicture} [scale=3.5]
\draw[->] (0.0, 0) -- (1.1, 0) node[below] {\small $s$};
\draw[->] (0.0,0.0) -- (0.0,1.1);
 \draw (0.9,1.2) node[left] { \footnotesize  {$N^{1/j}$\text{Re}$\left\{T^{j,k}_{N}f_{j}(s)\right\}$}};
\draw (1,0) node[below] {\small${1}$};
\draw (1,-0.001) --(1,0.003);
\draw (0.0,0.95) node[left] {\small${1}$};
\draw [blue,samples=500,domain=0.0:1] plot
(\x, {cos(0.65*\x r)});
\draw [red,samples=500,domain=0.0:1] plot
(\x, {cos(0.65*\x*\x r)});
 \draw(0.7,1) node[right,thick] {\tiny {$k=2$}};
 \draw(0.6,0.8) node[right,thick] {\tiny {$k=1$}};
 \draw(0.0,-0.15) node {};
\end{tikzpicture}\qquad
  \begin{tikzpicture} [scale=3.5]
\draw[->] (0.0, 0) -- (1.1, 0) node[below] {\small $s$};
\draw[->] (0.0,0.0) -- (0.0,1.1);
 \draw (0.9,1.2) node[left] { \footnotesize  {$N^{1/j}$\text{Im}$\left\{T^{j,k}_{N}f_{j}(s)\right\}$}};
\draw (1,0) node[below] {\small${1}$};
\draw (1,-0.001) --(1,0.003);
\draw (-0.001,1) --(0.003,1);
\draw (0.0,0.95) node[left] {\small${1}$};
\draw [blue,samples=500,domain=0.0:1] plot
(\x, {sin(0.65*\x r)});
\draw [red,samples=500,domain=0.0:1] plot
(\x, {sin(0.65*\x*\x r)});
 \draw(0.7,0.3) node[right,thick] {\tiny {$k=2$}};
 \draw(0.6,0.6) node[right,thick] {\tiny {$k=1$}};
  \draw(0.0,-0.15) node {};
\end{tikzpicture}\\
&\text{\small \emph{Both real and imaginary parts of the functions}}
\;T^{1,k}_{N}f_{1}\; \text{\small\emph{and}}\;
T^{2,k}_{N}f_{2}\; \text{\small \emph{have the same profile}}.
\end{align*}
(ii) \textbf{Constant data}\\
Let $\,g(x)=1.\,$
Whenever $\,\displaystyle s \in [0,\eta/N^{\frac{1}{k}}]\,$ we have
$\;\imath N{|x|}^{j}s^{k}\,=\,\mathcal{O}\left(\eta\right)\;$
for all $\;x\in B\;$ and it follows that
$\;\displaystyle e^{\imath N{|x|}^{j}s^{k}}=
1+\mathcal{O}\left(\eta\right).\;$
Hence, when $\,\displaystyle s \in [0,\eta/N^{\frac{1}{k}}],\,$
\begin{align*}
T^{j,k}_{N}g(s)\,=\,
\int_{B}\,e^{\imath N{|x|}^{j}s^{k}}\,dx\,=\,
\int_{B}\,\left(1+\mathcal{O}\left(\eta\right)\right)\,dx
\,\gtrsim 1.
\end{align*}
Therefore, recalling that $\eta$ is fixed,
\begin{equation}\label{lb20}
\int_{0}^{1}|T^{j,k}_{N}g(s)|^{q}\,ds
\,\geq\,
\int_{0}^{\eta/N^{\frac{1}{k}}}|T^{j,k}_{N}g(s)|^{q}\,ds
\,\gtrsim\,
\int_{0}^{\eta/N^{\frac{1}{k}}}\,ds
\,\approx\,N^{-\frac{1}{k}}.
\end{equation}
In view of (\ref{lb20}), we deduce that
\begin{align}\label{lb2}
\parallel T^{j,k}_{N}\parallel_{L^{p}\left(B\right)
\rightarrow L^{q}([0,1])}
\,\geq\,\frac{\parallel T^{j,k}_{N} g \parallel_{L^{q}([0,1])}}{
\parallel g \parallel_{L^{p}(B)}}
\;\gtrsim\;
N^{-\frac{1}{k}\frac{1}{q}}.
\end{align}
By rescaling, it is easy to verify that the estimate (\ref{lb2}) follows for any complex-valued constant function $g$. The figure below shows the behaviour of $ T_{N}^{j,k}g $ on $[0,1]$ in the dimension $n=1.$
\begin{align*}
\hspace*{-0.5 cm}
\begin{tikzpicture} [scale=10]
\draw[->] (0.0, 0.0) -- (0.48,0.0) node[below right] {$s$};
\draw[->] (0.0,0.0) -- (0.0,0.52);
\draw(0.25,0.44)  node[above] {\scriptsize $
    \text{Re}\left\{T^{1,k}_{N}g(s)\right\}=2$};
    \draw(0.45,0.44)  node[above] {\large $ \frac{\sin{(Ns^{k})}}{N\,s^{k}}$};
\draw (0.0,0.49) node[left] {\scriptsize ${2}$};
\draw (0.1,0.2) node[left] {\scriptsize $k=1$};
\draw (0.175,0.3) node[left] {\scriptsize $k=2$};
\draw [very thin,blue,samples=500,domain=0.001:0.45] plot
(\x, {2*sin((250*\x) r)/((1000*\x) )});
\draw [very thin,red,samples=500,domain=0.0012:0.45] plot
(\x, {2*sin((250*\x*\x) r)/((1000*\x*\x) )});
\end{tikzpicture}\qquad\qquad\qquad
\begin{tikzpicture} [scale=10]
\draw(0.23,0.44)  node[above] {\scriptsize $
    \text{Im}\left\{T^{1,k}_{N}g(s)\right\}=4$};
    \draw(0.46,0.44)  node[above] {\large $ \frac{\sin^{2}{(Ns^{k}/2)}}{N\,s^{k}}$};
\draw (0.0,0.49) node[left] {\scriptsize ${2}$};
\draw (0.0,0.49) --(0.001,0.49);
\draw (0.09,0.38) node[left] {\scriptsize $k=1$};
\draw (0.2,0.38) node[left] {\scriptsize $k=2$};
\draw [very thin,red,samples=500,domain=0.001:0.45] plot
(\x, {4*sin((100*\x*\x) r)*sin((100*\x*\x) r)/((800*\x*\x) )});
\draw [very thin,blue,samples=500,domain=0.0001:0.45] plot
(\x, {4*sin((50*\x) r)*sin((50*\x) r)/((400*\x) )});
\draw (-0.084, -0.084)  node[below] {};
\draw[->] (0.0, 0.0) -- (0.48,0.0) node[below right] {$s$};
\draw[->] (0.0,0.0) -- (0.0,0.52);
\end{tikzpicture}
\end{align*}
\begin{align*}
&\hspace*{-1.58 cm}\begin{tabular}{c c}
\hspace{-4.75 cm}\vspace{0.25 cm}   \scriptsize $2$& \\
\hspace{-1.7 cm}   \scriptsize $ k=2$&\\
\hspace{-3.3 cm}   \scriptsize $ k=1$&\\
&\hspace{-1 cm} \scriptsize $ k=1\quad k=2$
\vspace{-2.5 cm}\\
\hspace{0.5 cm}\scriptsize{ \text{Re}$\left\{T^{2,k}_{N}g(s)\right\}$}
&\hspace{2 cm}\scriptsize
{\text{Im}$\left\{T^{2,k}_{N}g(s)
    \right\}$} \vspace{-1.5 cm}\\
\includegraphics[scale=0.32]{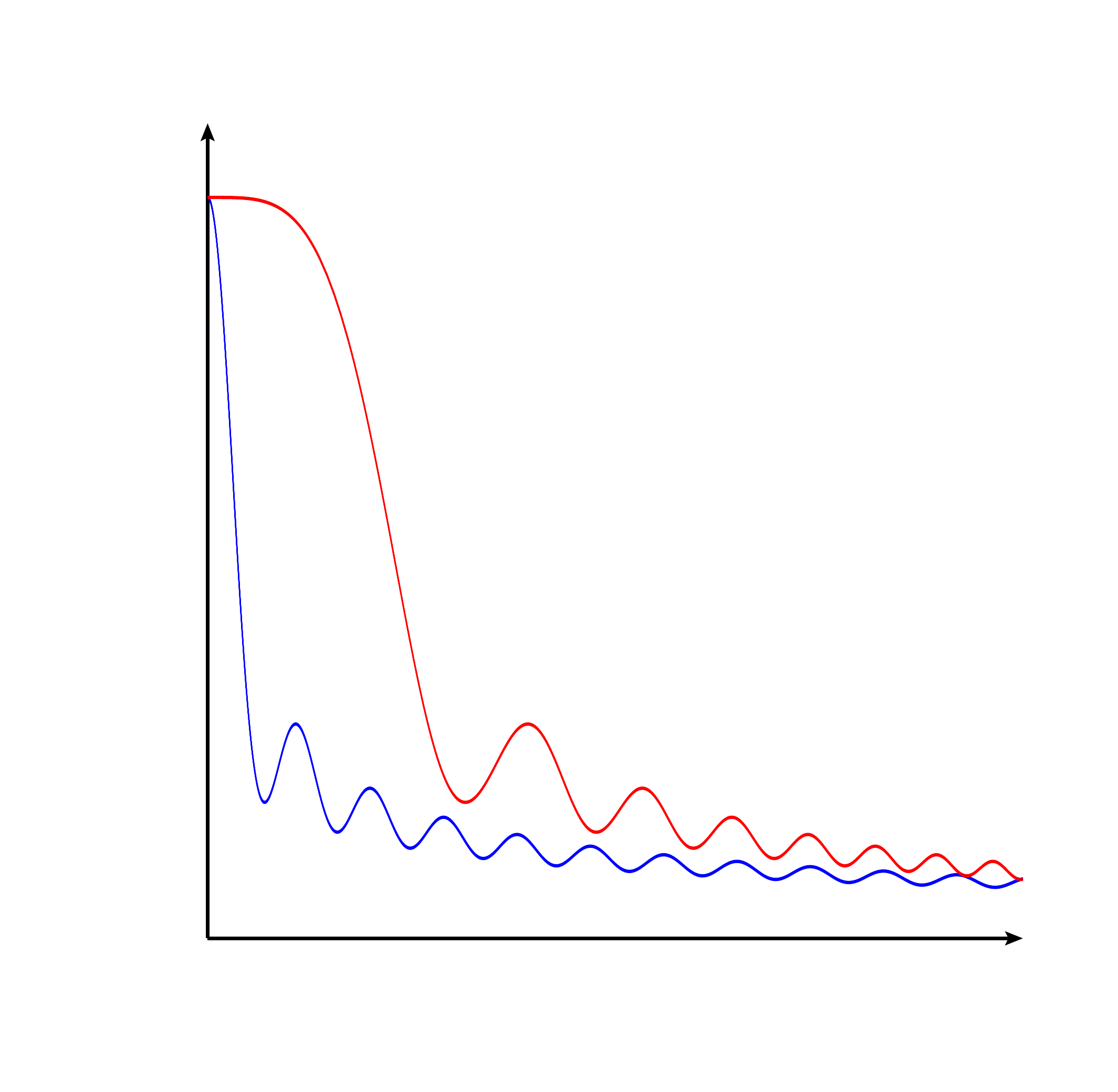}
&\qquad\quad\;
\includegraphics[scale=0.32]{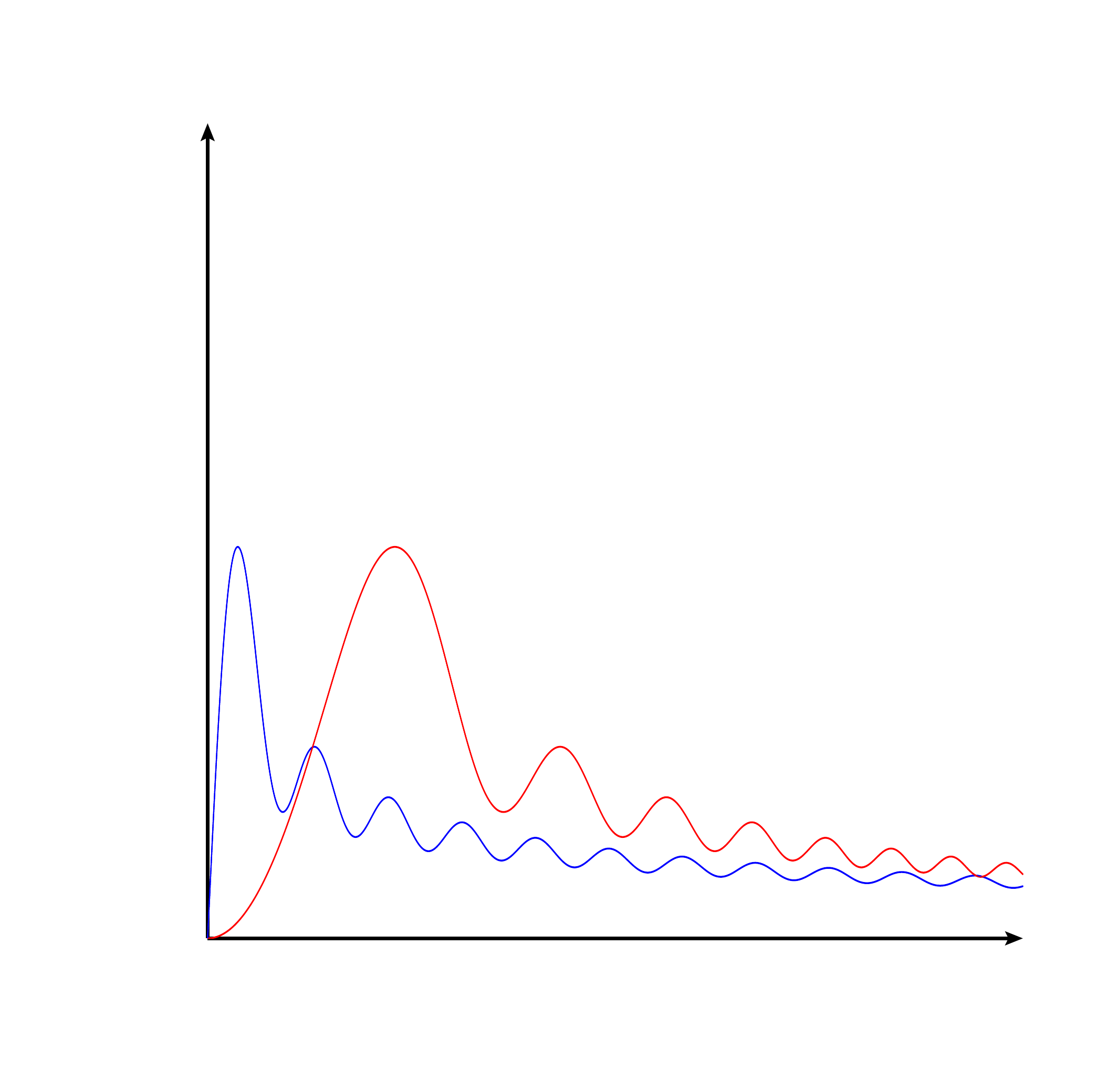}\vspace{-1 cm}\\
  \hspace{5.5 cm} \small$ s$&   \hspace{7 cm}
 \small $ s$
\end{tabular}\\
&\hspace{-2.3 cm} \text{\small\emph{Functions}}\; \text{\small \emph{Re}}\small \{ T^{1,k}_{N}g(s)\}\;
\text{\small\emph{vanish and}}\;\text{\small \emph{Re}}\small \{T^{2,k}_{N}g(s)\}\;
\text{\small \emph{change monotonicity, for the first time, when }}\;s=\sqrt[k]{\pi/N}
\end{align*}
(iii) \textbf{Oscillatory data}\\
Consider the oscillatory function
$\,h(x)=e^{2\imath N \left(|x|^2-|x|\right)}.\,$
Using polar coordinates we can write
\begin{align*}
T^{j,k}_{N}h(s)\,=\,
\int_{S^{n-1}}
\int_{0}^{1}\,e^{\imath N
\,\left({\rho}^{j}s^{k}+2\rho^2-2\rho\right)}\,
\rho^{n-1}\,d\rho \,d\omega\,=\,
\omega_{n-1}
\:I^{j,k}_{N}(s)
\end{align*}
where $I^{j,k}_{N}(s) $ is the oscillatory integral
given by
\begin{align}\label{inoscint}
I^{j,k}_{N}(s) =
\int_{0}^{1}\,e^{\imath N
\,\phi_{j,k}(\rho;s)}\,
\rho^{n-1}\,d\rho
\end{align}
with the phase $\displaystyle \phi_{j,k}(\rho;s) ={\rho}^{j}s^{k}+2\rho^2-2\rho. $\\
The quadratic function $\displaystyle \rho\rightarrow\phi_{j,k}(\rho;s)$, after a suitable translation along the vertical axis,
has a single nondegenerate stationary point
that happens to lie well inside $]\frac{1}{5},\frac{4}{5}[.$ Indeed, one can simply
write
\begin{align*}
\phi_{j,k}(\rho;s)=\left\{
                     \begin{array}{ll}
2\left(\rho-\frac{2-s^{k}}{4}\right)^{2}-
\frac{\left(2-s^{k}\right)^{2}}{8}, & \hbox{$j=1$;} \\
\left(2+s^{k}\right)\left(\rho-\frac{1}{2+s^{k}}\right)^{2}-
\frac{1}{\left(2+s^{k}\right)^{2}}, & \hbox{$j=2$.}
                     \end{array}
                   \right.
\end{align*}
Notice also that $\, \left(2-s^{k}\right)/4\in[\frac{1}{4},\frac{1}{2}]\,$
and $\,\left(2+s^{k}\right)^{-1}\in[\frac{1}{3},
\frac{1}{2}]\,$
when $\,s\in [0,1].$
In fact, this is what we were after when we
used the oscillatory function $h$ with its
particular quadratic phase. Let us see how we benefit from this. We shall work on the integral $\,I^{1,k}_{N}(s)\,$ and the applicability
 of the same procedure to the integral
$\,I^{2,k}_{N}(s)\,$ will be obvious.
For simplicity, let
$z$ denote $\,\left(2-s^{k}\right)/4.\,$
Then
 \begin{align}
\nonumber
e^{2\imath N\,z^2}\,I^{1,k}_{N}(s) =&
\,\int_{0}^{1}\,e^{2\imath N
\,\left(\rho-z\right)^{2}}\,
\rho^{n-1}\,d\rho\\
\label{feq}=&\,z^{n-1}\,\int_{0}^{1}\,e^{2\imath N
\,\left(\rho-z\right)^{2}}\,d\rho+
\int_{0}^{1}\,e^{2\imath N
\,\left(\rho-z\right)^{2}}\,
\left(\rho^{n-1}-z^{n-1}\right)\,d\rho.
\end{align}
We compute
\begin{align}\label{t1}
\hspace{-1 cm}
\int_{0}^{1}\,e^{2\imath N
\,\left(\rho-z\right)^{2}}\,d\rho=
\int_{-\infty}^{+\infty}\,e^{2\imath N
\,\left(\rho-z\right)^{2}}\,d\rho-
\int_{-\infty}^{0}\,e^{2\imath N
\,\left(\rho-z\right)^{2}}\,d\rho-
\int_{1}^{+\infty}\,e^{2\imath N
\,\left(\rho-z\right)^{2}}\,d\rho.
\end{align}
Using the identity (See Exercise 2.26 in \cite{taobook})
\begin{align*}
\int_{-\infty}^{+\infty}
\,e^{-ax^2}\,e^{bx}\,dx=\sqrt{\frac{\pi}{a}}
\,e^{b^2/4a},\quad a,b \in \mathbb{C},\;
 \textrm{Re}(a) >0 \qquad \text{we get}
\end{align*}
\begin{align}\label{11}
\hspace{-0.5 cm}\int_{-\infty}^{+\infty}\,e^{2\imath N
\,\left(\rho-z\right)^{2}}\,d\rho
=\sqrt{\frac{\pi}{2N}}\,e^{\frac{\pi}{4}\imath}.
\end{align}
And since
\begin{align*}
\hspace*{-1 cm}
\left|\;\int_{-\infty}^{0}\,e^{2\imath N
\,\left(\rho-z\right)^{2}}\,\partial_{\rho}
\left(\rho-z\right)^{-1}\,d\rho \right|\,\leq\,
 \frac{1}{z},\quad \left|\;
\int_{1}^{+\infty}\,e^{2\imath N
\,\left(\rho-z\right)^{2}}\,\partial_{\rho}
\left(\rho-z\right)^{-1}\,d\rho \right|\,\leq\,
\frac{1}{1-z},
\end{align*}
then integration by parts implies
\begin{align}
\label{22}&\int_{-\infty}^{0}\,e^{2\imath N
\,\left(\rho-z\right)^{2}}\,d\rho=
\frac{\imath\, e^{2\imath N z^{2}}}{4 N z}
+\mathcal{O}\left(\frac{1}{Nz}\right),\\
\label{33}&\int_{1}^{+\infty}\,e^{2\imath N
\,\left(\rho-z\right)^{2}}\,d\rho=
\frac{\imath\, e^{2\imath N\left(1-z\right)^{2}}}{4 N\left(1-z\right)}
+\mathcal{O}\left(\frac{1}{N\left(1-z\right)}\right).
\end{align}
Recalling that $\,\frac{1}{4}\leq z\leq \frac{1}{2}\;$
and using (\ref{11}), (\ref{22}), (\ref{33})
in (\ref{t1}) we obtain
\begin{align}\label{44}
\int_{0}^{1}\,e^{2\imath N
\,\left(\rho-z\right)^{2}}\,d\rho\,=\,
\sqrt{\frac{\pi}{2N}}\,e^{\frac{\pi}{4}\imath}
+\mathcal{O}\left(\frac{1}{N}\right).
\end{align}
This gives us an estimate for the first integral
on the right hand side of (\ref{feq}).
The second integral is $\;\mathcal{O}\left({1}/{N}\right).\;$ This
follows from integration by parts and the smoothness
of the polynomial $\;P(\rho;z):={\left(\rho^{n-1}-z^{n-1}\right)}/{\left(\rho-z\right)}=
\sum_{\ell=0}^{n-2}\,\rho^{n-2-\ell}\,z^{\ell}\;$
as we can write
\begin{align*}
\int_{0}^{1}\,e^{2\imath N
\,\left(\rho-z\right)^{2}}\,
\left(\rho^{n-1}-z^{n-1}\right)\,d\rho\,=\,
\frac{1}{4\imath N}
\int_{0}^{1}\,
P(\rho;z)\,
\partial_{\rho}\,e^{2\imath N
\,\left(\rho-z\right)^{2}}\,d\rho.
\end{align*}
Plugging (\ref{44}) together with the latter estimate into (\ref{feq}) we get that
\begin{align}\label{sm}
e^{2\imath N\,z^2}\,I^{1,k}_{N}(s)
\,=\,z^{n-1}\,
\sqrt{\frac{\pi}{2N}}\,e^{\frac{\pi}{4}\imath}
+\mathcal{O}\left(\frac{1}{N}\right).
\end{align}
From (\ref{sm}) follows the estimate
\begin{align*}
\left|I^{1,k}_{N}(s)\right|
\,\gtrsim\,N^{-1/2}.
\end{align*}
An explanation for the estimate above comes from the fact that the function
$\,\lambda_{N}(\rho;z)=
\cos{\left(2N\,\left(\rho-z\right)^{2}\right)}\,$
remains positive for $\;|\rho-z|<\sqrt{\left(\pi/4N\right)}\;$
and the further we move from the stationary point
$\rho=z$ it, unlike the slowly varying factor $\rho^{n-1}$, oscillates rapidly for large $N$
so that, when summing over $\rho$, integrals over neighbouring halfwaves where
$\lambda_{N}$  changes sign almost cancel.
See the figure below.
An identical estimate for
$\,I^{2,k}_{N}(s)\,$
follows applying the same argument above.
The approach adopted here is
standard. It represents the key idea
of the proof of the stationary phase method
illustrated by Lemma \ref{stationary}.
\begin{align*}
\begin{tikzpicture}[yscale=1.5]
\fill[fill = black!50] (3*pi/8,0) -- plot [domain=3*pi/8:11*pi/13] (\x,{cos(64*\x*\x )}) -- (11*pi/13,0) -- cycle;
\fill[fill = black!50] (-3*pi/8,0) -- plot [domain=-3*pi/8:-11*pi/13] (\x,{cos(64*\x*\x )}) -- (-11*pi/13,0) -- cycle;
\fill[fill = black!25] (11*pi/13,0) -- plot [domain=11*pi/13:235*pi/208] (\x,{cos(64*\x*\x )}) -- (235*pi/208,0) -- cycle;
\fill[fill = black!25] (-11*pi/13,0) -- plot [domain=-11*pi/13:-235*pi/208] (\x,{cos(64*\x*\x )}) -- (-235*pi/208,0) -- cycle;
\fill[fill = black!5] (235*pi/208,0) -- plot [domain=235*pi/208:141*pi/104] (\x,{cos(64*\x*\x )}) -- (141*pi/104,0) -- cycle;
\fill[fill = black!5] (-235*pi/208,0) -- plot [domain=-235*pi/208:-141*pi/104] (\x,{cos(64*\x*\x )}) -- (-141*pi/104,0) -- cycle;
\draw [ <->] (-6.5,0) -- (6.5,0);
\draw [help lines,dashed,<-] (0,1.3) -- (0,0);
\draw (0,0) node[below] {$\rho=z$};
\draw (-5,1.3) node[above]{$ \cos{\left(N\,\left(\rho-z\right)^{2}\right)}$}; ;
\draw [thick,samples=500,domain=-2*pi:2*pi] plot
(\x, {cos(64*\x*\x )});
\draw [ <-](-3*pi/8+0.01,-0.7)
--(-2*pi/8+0.18,-0.7);
\draw [ ->](2*pi/8-0.1,-0.7)
--(3*pi/8-0.01,-0.7);
\draw (6.5,0) node[right] {$\rho$};
\draw (0,-0.7) node {$\sqrt{{\pi}/{2N}}$};
\draw [help lines,dashed] (-3*pi/8,1) -- (-3*pi/8,-1);
\draw [help lines,dashed] (3*pi/8,1) -- (3*pi/8,0-1);
\end{tikzpicture}
\end{align*}
Finally, since
$\;\displaystyle \parallel h \parallel_{L^{p}(B)}\,=\, |B|^{{1}/{p}}\,\approx\,1,\;$ then
\begin{align}\label{lb3}
\parallel T^{j,k}_{N}\parallel_{L^{p}\left(B\right)
\rightarrow L^{q}([0,1])}
\,\geq\,\frac{\parallel T^{j,k}_{N} h \parallel_{L^{q}([0,1])}}{
\parallel h \parallel_{L^{p}(B)}}
\;\gtrsim\;
N^{-\frac{1}{2}}.
\end{align}
Putting (\ref{lb1}),  (\ref{lb2}) and
(\ref{lb3}) together we deduce
\begin{align*}
 \parallel T^{j,k}_{N} \parallel_{L^{p}(B)
\rightarrow L^{q}([0,1])}
\;\;\gtrsim\;
N^{-\min\left\{\frac{n}{j}\left(1-\frac{1}{p}\right),
\,\frac{1}{k}\frac{1}{q},\,\frac{1}{2}\right\}}
\,=\,N^{- C_{j,k,n}\left(\frac{1}{p},\frac{1}{q}\right)}.
\end{align*}
\textbf{\emph{\underline{Step 2}}.}
The $\,L^{2} - L^{2}\,$ estimate takes the form:
\begin{eqnarray}\label{energy}
\left.
             \begin{array}{ll}
\vspace{0.3 cm}
\parallel T^{j,k}_{N}f\parallel_{L^{2}([0,1])}
\;\lesssim \; N^{-1/2k}\,\parallel f \parallel_{L^{2}\left(B\right)}, & \hbox{$n\geq j$,} \\
               \parallel T^{2,k}_{N}f\parallel_{L^{2}([0,1])}
\;\lesssim \; N^{-n/2j}\,\parallel f \parallel_{L^{2}\left(B\right)}, & \hbox{$n=1$.}
             \end{array}
           \right \}
\end{eqnarray}
 Besides (\ref{lb2}), the estimate (\ref{energy}) demonstrates the difference between linear ($k=1$) and quadratic ($k=2$) oscillations.
Let $\,x\in {\mathbb{R}}^{n}-\{0\}.\,$
The phase $\;s \longrightarrow  {|x|}^{j}\,s^{k}\;$
of the oscillatory factor in (\ref{intop})
is non-stationary when $\,k=1.\,$ While
in the case $\,k=2,\,$
it is stationary with the nondegenerate
critical point $s=0.\,$ This is where
non-stationary and stationary phase methods
(see lemmas \ref{nonstationary} and
\ref{stationary0} below)
for estimating oscillatory integrals come into play.
As expected from (\ref{lb1}), the role of $j$
appears only in the dimension $n=1.$
Using the estimate (\ref{energy}) in Lemma \ref{holder}
we infer
\begin{align*}
 \parallel T^{j,k}_{N} \parallel_{L^{p}(B)
\rightarrow L^{q}([0,1])}
\;\;\lesssim\; N^{- C_{j,k,n}\left(\frac{1}{p},\frac{1}{q}\right)}.
\end{align*}
\subsubsection*{3. Proof of the energy estimate
 (\ref{energy})}
To prove the estimate (\ref{energy})
we need lemmas \ref{kernelsk}, \ref{kernelsq} and \ref{even} that we give below.
Lemma \ref{kernelsk} is based on the assertions
of lemmas \ref{nonstationary} and
\ref{stationary0}.
\begin{lem}\label{nonstationary}
(\cite{stein}, Proposition 1 Chapter VIII)
Let $\,\psi \in C^{\infty}_{c}\left(\mathbb{R}
\right)\,$ and let $\displaystyle\; I(\lambda)=
\int_{\mathbb{R}}\,\psi(s)\,e^{\imath \,\lambda\,s}\,ds. \,$ Then $\;\displaystyle |I(\lambda)|\;\lesssim\;
\min{\left\{
\frac{1}{1+|\lambda|},
\frac{1}{1+\lambda^{2}}\right\}}.$
\end{lem}
Observing that $\;\displaystyle \int_{0}^{1}\,e^{\imath \,\lambda\,s^{2}}\,ds\,=\,
\frac{1}{2}\int_{-1}^{1}\,e^{\imath \,\lambda\,s^{2}}\,ds\;$ and arguing as
in (\ref{t1})-(\ref{44}) implies the estimate
in Lemma \ref{stationary0}.
\begin{lem}\label{stationary0}
\begin{align*}
\left|\int_{0}^{1}\,e^{\imath \,\lambda\,s^{2}}\,ds\right|\;\lesssim \; \max{\left\{\frac{1}{1+\sqrt{|\lambda|}},
\frac{1}{1+|\lambda|}\right\}}.
\end{align*}
\end{lem}
\begin{lem}\label{kernelsk}
Let $\,\psi \in C^{\infty}_{c}\left(\mathbb{R}
\right)\,$ and let
$\;K_{N}^{j,k}:{\mathbb{R}}^{n}\times{\mathbb{R}}^{n}
\longrightarrow {\mathbb{C}}\;$ be defined by
\begin{align*}
K_{N}^{j,k}(x,y):=
\left\{
  \begin{array}{ll}
 \displaystyle   \int_{\mathbb{R}}\,\psi(s)\,e^{\imath N \left({|x|}^{j}-{|y|}^{j}\right)s}\,ds, & \hbox{$k=1$;} \\\\
 \displaystyle  \int_{0}^{1}\,\,e^{\imath N \left({|x|}^{j}-{|y|}^{j}\right)s^{2}}\,ds  , & \hbox{$k=2$.}
  \end{array}
\right.
\end{align*}
Then
\begin{align}
\label{kernelsk1}
&\hspace*{-1 cm}
|K_{N}^{j,1}(x,y)|\;\lesssim\;
\min{\left\{
\left(1+N\,\left|{|x|}^{j}-{|y|}^{j}\right|\right)^{-1},
\left(1+N^{2} \, \left({|x|}^{j}-{|y|}^{j}\right)^{2}\right)^{-1}
\right\}},
\\ \label{kernelsk2}
&\hspace*{-1 cm} |K_{N}^{j,2}(x,y)|\;\lesssim\;
\max{\left\{\left( 1+\sqrt{N}\, \sqrt{\left|{|x|}^{j}-{|y|}^{j}\right|}\right)^{-1},
\left(1+N\, \left|{|x|}^{j}-{|y|}^{j}\right|\right)^{-1}\right\}}.
\end{align}
\end{lem}
The next lemma is mainly a consequence of
Young's inequality.
\begin{lem}\label{kernelsq}
Let $\,p,q,r\geq 1\,$ and $\,1/p+1/q+1/r =2.\,$
Let $\,f\in L^{p}(B),\,$ $\,g\in L^{q}(B)\,$
and $\,h\in L^{r}([0,1]).\,$  Then
\begin{align*}
 \left|\,\int_{{B}}\,\int_{{B}}\,
f(x)\,f(y)\,h(|x|^{m}-|y|^{m})\,dx\,dy\,\right|\;\lesssim
\;\parallel f \parallel_{L^{p}(B)}\,
\parallel g \parallel_{L^{q}(B)}\,
\parallel h \parallel_{L^{r}([0,1])}
\end{align*}
provided $\,m\leq n$.
\end{lem}
\begin{proof}
Switching to polar coordinates by setting
$\,x=r_{1}\theta_{1}\,$ and $\,y=r_{2}\theta_{2}\,$
then applying Fubini's theorem gives
\begin{align}\label{newlemma1}
 \left|\,\int_{{B}}\,\int_{{B}}\,
f(x)\,f(y)\,h(|x|^{m}-|y|^{m})\,dx\,dy\,\right|
\,\leq\,\int_{S^{n-1}}\,\int_{S^{n-1}}\,
|Q(\theta_{1},\theta_{2})|
\,d\theta_{1}\,d\theta_{2}
\end{align}
where
\begin{align*}
Q(\theta_{1},\theta_{2})\,=\,
\int_{0}^{1}\,\int_{0}^{1}\,
f(r_{1}\theta_{1})\,g(r_{2}\theta_{2})
\,h\left({r_{1}}^{m}-{r_{2}}^{m}\right)
\,r_{1}^{n-1}\,r_{2}^{n-1}\,dr_{1}\,dr_{2}.
\end{align*}
Changing variables
$\:r_{i}^{m}\,\longrightarrow\, \rho_{i}\:$
then using Young's inequality we get
\begin{align*}
\hspace*{-1 cm}
|Q(\theta_{1},\theta_{2})|\,\lesssim\,
\left(\int_{0}^{1}\left|f(\sqrt[m]{\rho_{1}}\,\theta_{1})
\right|^{p}\,\rho_{1}^{p\frac{n-m}{m}}\,d\rho_{1}\right)
^{\frac{1}{p}}
\left(\int_{0}^{1}\left|g(\sqrt[m]{\rho_{2}}\,\theta_{2})
\right|^{q}\,\rho_{2}^{q\frac{n-m}{m}}\,d\rho_{2}\right)
^{\frac{1}{q}}
\parallel h \parallel_{L^{r}([0,1])}.
\end{align*}
Reversing the variables change
in the first two integrals on the right-hand side of the latter estimate we obtain
\begin{align}\label{newlemma2}
\hspace*{-1 cm}
\nonumber |Q(\theta_{1},\theta_{2})|\,\lesssim&\,
\left(\int_{0}^{1}\left|f({r_{1}}\,\theta_{1})
\right|^{p}\,r_{1}^{(p-1)(n-m)}\,
r_{1}^{n-1}\,dr_{1}\right)
^{\frac{1}{p}}\\&\;\nonumber
\left(\int_{0}^{1}\left|g({r_{2}}\,\theta_{2})
\right|^{q}\,r_{2}^{(p-1)(n-m)}\,
r_{2}^{n-1}\,dr_{2}\right)
^{\frac{1}{q}}\,
\parallel h \parallel_{L^{r}([0,1])}\\
\leq&
\left(\int_{0}^{1}\left|f({r_{1}}\,\theta_{1})
\right|^{p}\,r_{1}^{n-1}\,dr_{1}\right)
^{\frac{1}{p}}
\left(\int_{0}^{1}\left|g({r_{2}}\,\theta_{2})
\right|^{q}\,r_{2}^{n-1}\,dr_{2}\right)
^{\frac{1}{q}}
\,\parallel h \parallel_{L^{r}([0,1])}
\end{align}
as long as $\,m\leq n.$ Invoking H\"{o}lder's inequality it follows that
\begin{align}
\nonumber  &\int_{S^{n-1}}\,
\left(\int_{0}^{1}\left|f({r_{1}}\,\theta_{1})
\right|^{p}\,
r_{1}^{n-1}\,dr_{1}\right)^{\frac{1}{p}}\,d\theta_{1}\\
\label{newlemma3} &\hspace{0.8 cm}\leq\;\omega_{n-1}^{1-\frac{1}{p}}\;
 \left( \int_{S^{n-1}}\,\int_{0}^{1}
\left|f({r_{1}}\,\theta_{1})
\right|^{p}\,r_{1}^{n-1}\,dr_{1}\,d
\theta_{1}\right)^{\frac{1}{p}}\;=
\;\omega_{n-1}^{1-\frac{1}{p}}\;
\parallel f \parallel_{L^{p}(B)},\\
\nonumber  &
\int_{S^{n-1}}\,
\left(\int_{0}^{1}\left|g({r_{2}}\,\theta_{2})
\right|^{q}\,
r_{2}^{n-1}\,dr_{2}\right)^{\frac{1}{q}}\,d\theta_{2}\\
\label{newlemma4} &
\hspace{0.8 cm}\leq\;\omega_{n-1}^{1-\frac{1}{q}}\;
 \left( \int_{S^{n-1}}\,\int_{0}^{1}
\left|g({r_{2}}\,\theta_{2})
\right|^{q}\,r_{2}^{n-1}\,dr_{2}\,d
\theta_{2}\right)^{\frac{1}{q}}\;=
\;\omega_{n-1}^{1-\frac{1}{q}}\;
\parallel g \parallel_{L^{q}(B)}.
\end{align}
Returning to (\ref{newlemma1})
with the estimates (\ref{newlemma2}),
(\ref{newlemma3}) and (\ref{newlemma4}) concludes the proof.
\end{proof}
Remark \ref{even0} together with Lemma \ref{homogeneous} are needed to show Lemma \ref{even}.
\begin{rem}\label{even0}
Suppose that the integral
\begin{align*}
 J\,=\,
\int_{-b_{1}}^{b_{1}}...\int_{-b_{m}}^{b_{m}}
\,K(t_{1},...,t_{m})\,
f_{1}(t_{1})...f_{m}(t_{m})\,dt_{1}...dt_{m}
\end{align*}
exists. If $\,K\,$
is even in all its variables then
\begin{align*}
 J\,=\,
\int_{0}^{b_{1}}...
\int_{0}^{b_{m}}
\,K(t_{1},...,t_{m})\,
\prod_{i=1}^{m}\left(f_{i}(t_{i})+f_{i}(-t_{i})\right)
\,dt_{1}...dt_{m}.
\end{align*}
This follows easily from
the fact that the integrand
in the second expression for $\,J\,$
is even in all variables.
\end{rem}
Lemma \ref{homogeneous} discusses the boundedness
of a bilinear form with a homogeneous kernel.
\begin{lem}\label{homogeneous}
Let $\,f\in L^{p}([0,1])\,$ and
$\,g\in L^{q}([0,1])\,$
with $\,1\leq p \leq +\infty\,$ and $\,1/p\, +\, 1/q=1.\,$
Assume that $\,K:{[0,1]}\times
{[0,1]}\longrightarrow {\mathbb{R}}\,$
is homogeneous of degree $-1,\,$
that is, $\,K(\lambda x, \lambda y)=
\lambda^{-1} K(x,y),\,$ for $\,\lambda>0.\,$
Assume also that
\begin{align*}
 \int_{0}^{+\infty}
\,\left|K(x,1)\right|\,{x}^{-\frac{1}{p}}\,dx
\,\lesssim\,1 \qquad \text{or } \qquad
 \int_{0}^{+\infty}
\,\left|K(1,y)\right|\,{y}^{-\frac{1}{q}}\,dy
\,\lesssim\,1.
\end{align*}
Then
\begin{align*}
\left|\int_{0}^{1}\int_{0}^{1}\,
K(x,y)\,f(x)\,g(y)\,dx\,dy\right|\;\lesssim\;
\parallel f \parallel_{L^{p}([0,1])}\,
\parallel g \parallel_{L^{q}([0,1])}.
\end{align*}
 \end{lem}
In \cite{hardy}, one can find a proof
for the case when the integrals
that define the bilinear form are taken over $\,[0,+\infty[.\,$ We treat this slightly trickier
case of finite range without using the result in
\cite{hardy}.
\begin{proof}
Let $\,\displaystyle Q(f,g)\,=\,
\int_{0}^{1}\int_{0}^{1}\,
K(x,y)\,f(x)\,g(y)\,dx\,dy.\,$
Using a change of variables,
$\,x\rightarrow y.u,\,$ and exploiting the homogeneity
of the kernel we have
\begin{align*}
\hspace*{-0.8 cm} Q(f,g)\,=\,
\int_{0}^{1}y\,g(y)\int_{0}^{\frac{1}{y}}
K(y.u,y)\,f(y.u)\,du\,dy\,=\,
\int_{0}^{1}g(y)\int_{0}^{\frac{1}{y}}
K(u,1)\,f(y.u)\,du\,dy.
\end{align*}
By Fubini's theorem we may write
 \begin{align}\label{qfg}
\hspace*{-1 cm}
Q(f,g)=\int_{0}^{1} K(u,1) \int_{0}^{1} f(y.u)\,g(y)\,dy\,du+
\int_{1}^{+\infty} K(u,1) \int_{0}^{\frac{1}{u}} f(y.u)\,g(y)\,dy\,du.
\end{align}
But by  H\"{o}lder's inequality we have
\begin{align*}
\hspace*{-0.8 cm}
\left|\int_{0}^{1}\,f(y.u)\,g(y)\,dy\right|
\;\leq&\;\left(\int_{0}^{1}\,|f(y.u)|^{p}\,dy\right)
^{\frac{1}{p}}
\left(\int_{0}^{1}\,|g(y)|^{q}\,dy\right)^{\frac{1}{q}}\\
\,&\hspace*{-2 cm}=u^{-\frac{1}{p}}\,
\left(\int_{0}^{u}\,|f(x)|^{p}\,dx\right)
^{\frac{1}{p}}\,
\parallel g \parallel_{L^{q}([0,1])}
\;\leq\;u^{-\frac{1}{p}}\,\parallel f \parallel_{L^{q}([0,1])}\,
\parallel g \parallel_{L^{q}([0,1])}
\end{align*}
for all $\,0 < u < 1.\,$ Similarly
\begin{align*}
\hspace*{-0.8 cm}
\left|\int_{0}^{\frac{1}{u}}\,f(y.u)\,g(y)\,dy\right|
\;\leq&\;\left(\int_{0}^{\frac{1}{u}}
\,|f(y.u)|^{p}\,dy\right)
^{\frac{1}{p}}\,
\left(\int_{0}^{\frac{1}{u}}\,|g(y)|^{q}
\,dy\right)^{\frac{1}{q}}\\
\,&\hspace*{-3.4 cm}=u^{-\frac{1}{p}}\,
\left(\int_{0}^{1}\,|f(x)|^{p}\,dx\right)
^{\frac{1}{p}}\,
\left(\int_{0}^{\frac{1}{u}}\,|g(y)|^{q}
\,dy\right)^{\frac{1}{q}}
\;\leq\;u^{-\frac{1}{p}}\,\parallel f \parallel_{L^{q}([0,1])}\,
\parallel g \parallel_{L^{q}([0,1])}
\end{align*}
for all $\,1< u < +\infty.\,$ Using the last two inequalities together
with the triangle inequality  in (\ref{qfg})
we get
 \begin{align*}
\hspace*{-1 cm}
|Q(f,g)|\leq& \,
\,\parallel f \parallel_{L^{q}([0,1])}\,
\parallel g \parallel_{L^{q}([0,1])}\,
\left(\int_{0}^{1} |K(u,1)| \,u^{-\frac{1}{p}}\,du+
\int_{1}^{+\infty} |K(u,1)|\,u^{-\frac{1}{p}}\,du
\right)\\
 \lesssim&\;
\parallel f \parallel_{L^{q}([0,1])}\,
\parallel g \parallel_{L^{q}([0,1])},
\qquad \text{when} \quad
\int_{0}^{+\infty} |K(x,1)|\,x^{-\frac{1}{p}}\,dx
\,\lesssim\,1.
\end{align*}
When $\;\displaystyle \int_{0}^{+\infty}\left|K(1,y)\right|
{y}^{-\frac{1}{q}}dy\lesssim 1\;$
the assertion follows analogously.
\end{proof}
\begin{rem}
If $\,K(x,y)=\left( x + y \right)^{-1}\,$
in Lemma \ref{homogeneous} we get Hilbert's
inequality.
\end{rem}
\begin{lem}\label{even}
Let $\,f,g \in L^2([-1,1]).$ Then
\begin{align}
\label{even1}&\int_{-1}^{1}\,\int_{-1}^{1}\,
\frac{|f(x)||g(y)|}{1+N\,
\left|x^2-y^2\right|}
\,dx\,dy \;\lesssim\; \frac{1}{\sqrt{N}}\,
\parallel f \parallel_{L^{2}([-1,1])}\,\parallel g \parallel_{L^{2}([-1,1])}, \\
\label{even2} &\int_{-1}^{1}\,\int_{-1}^{1}\,
\frac{|f(x)|\,|g(y)|}{
\sqrt{\left|{x}^{2}-{y}^{2}\right|}}
\,dx\,dy \;\lesssim\; \parallel f \parallel_{L^{2}([-1,1])}\,\parallel g \parallel_{L^{2}([-1,1])}.
\end{align}
\end{lem}
\begin{proof}
Beginning with the estimate (\ref{even1}),
Remark \ref{even0} suggests estimating\\\\
$\; \displaystyle \int_{0}^{1}\,\int_{0}^{1}\,
\frac{|f(\pm x)||g(\pm y)|}{1+N\,
\left|x^2-y^2\right|}
\,dx\,dy.\;$ Let $\;\displaystyle W_{N}(f,g):=
\int_{0}^{1}\,\int_{0}^{1}\,
\frac{|f(x)||g(y)|}{1+N\,
\left|x^2-y^2\right|}
\,dx\,dy.$\\\\
If $\,x,y \geq 0\,$ and
$\,|x-y|>>1/\sqrt{N}\,$ then we also have
$\,x+y>>1/\sqrt{N}\,$ and consequently $\,N\left|x^2-y^2\right|>>1.\,$
Therefore
\begin{align}
\hspace{-0.8 cm}
\nonumber W_{N}(f,g)&\approx
\int\,\int_{
\substack{0\leq x,y\leq1,\\
|x-y|\lesssim\; 1/\sqrt{N}}}
\frac{|f(x)||g(y)|}{1+N\,
\left|x^2-y^2\right|}
\,dx\,dy+
\int\,\int_{
\substack{0\leq x,y\leq1,\\
|x-y|>> 1/\sqrt{N}}}
\frac{|f(x)||g(y)|}{1+N\,
\left|x^2-y^2\right|}
\,dx\,dy\\
\nonumber &\lesssim
\int\,\int_{
\substack{0\leq x,y\leq1,\\
|x-y|\lesssim\; 1/\sqrt{N}}}
{|f(x)||g(y)|}\,dx\,dy+
\frac{1}{N}\,\int\,\int_{
\substack{0\leq x,y\leq1,\\
|x-y|>> 1/\sqrt{N}}}
\frac{|f(x)||g(y)|}{\left|x^2-y^2\right|}
\,dx\,dy\\
\label{h1} &\lesssim
\int_{0}^{1}\,\int_{0}^{1}\,
\chi_{N}{\left(|x-y|\right)}{|f(x)||g(y)|}\,dx\,dy+
\frac{1}{\sqrt{N}}\,\int_{0}^{1}\,\int_{0}^{1}\,
\frac{|f(x)||g(y)|}{x+y}
\,dx\,dy
\end{align}
where $\,\chi_{N}\,$ is the characteristic function
of the interval $\,[0,1/\sqrt{N}\,].$
By Young's inequality we have
\begin{align}\label{h2}
\int_{0}^{1}\,\int_{0}^{1}\,
\chi_{N}{\left(|x-y|\right)}{|f(x)||g(y)|}
\,dx\,dy\,\leq\,
\frac{1}{\sqrt{N}}\,
\parallel f \parallel_{L^{2}([0,1])}\,\parallel g \parallel_{L^{2}([0,1])}.
\end{align}
And by Hilbert's inequality
\begin{align}\label{h3}
\int_{0}^{1}\,\int_{0}^{1}\,
\frac{|f(x)||g(y)|}{x+y}
\,dx\,dy\,\lesssim \;\parallel f \parallel_{L^{2}([0,1])}\,\parallel g \parallel_{L^{2}([0,1])}.
\end{align}
Using (\ref{h2}) together with
(\ref{h3}) in (\ref{h1}) we obtain
\begin{align*}
\int_{0}^{1}\,\int_{0}^{1}\,
\frac{|f(x)||g(y)|}{1+N\,
\left|x^2-y^2\right|}
\,dx\,dy\,\lesssim\,\frac{1}{\sqrt{N}}\,
\parallel f \parallel_{L^{2}([0,1])}\,\parallel g \parallel_{L^{2}([0,1])}.
\end{align*}
In obtaining (\ref{h1}), we worked only on the kernel of $W_{N}.$ It is therefore easy to see that replacing the function $\,x\rightarrow f(x)\,$
by the function $\,x\rightarrow f(-x)\,$
or $\,y\rightarrow g(y)\,$
by $\,y\rightarrow g(-y)\,$
then repeating the routine above
eventually leads to the estimate
\begin{align*}
\int_{0}^{1}\,\int_{0}^{1}\,
\frac{|f(\pm x)||g(\pm y)|}{1+N\,
\left|x^2-y^2\right|}
\,dx\,dy\,\lesssim\,\frac{1}{\sqrt{N}}\,
\parallel f \parallel_{L^{2}([-1,1])}\,\parallel g \parallel_{L^{2}([-1,1])}.
\end{align*}
This proves (\ref{even1}). Taking advantage of Remark \ref{even0} again and arguing like before,
it suffices to\\\\ estimate
$\displaystyle
 V(f,g)=\int_{0}^{1}\,\int_{0}^{1}\,
\frac{|f(x)|\,|g(y)|}{\sqrt{\left|{x}^{2}-{y}^{2}\right|}}
\,dx\,dy.\;$  Since $\; \displaystyle \int_{0}^{+\infty}\frac{dz}{\sqrt{z}\,\sqrt{|1-z^2|}}
\,\approx\,1,$\\\\
a direct application of Lemma \ref{homogeneous} then
gives $\,V(f,g)\,\lesssim\,\parallel f \parallel_{L^{2}([0,1])}\,\parallel g \parallel_{L^{2}([0,1])}.$
\end{proof}
We are now ready to prove (\ref{energy}).
We do this for each of the cases
$k=1$ and $k=2$ separately.\\
\textbf{The phase is linear in $\textbf{s}\,$ $\,(k=1)$}:\\
Let $\psi$ be a nonnegative smooth cutoff function
such that $\,{supp}\:\psi \subset\;]-1,2[\,$
and $\,\psi(s)=1\,$ on $\,[0,1]$.
Since $\,|T^{j,1}_{N} f |^2 \,=\, T^{j,1}_{N} f\;\;\overline{T^{j,1}_{N} f}.\,$ Then
\begin{align*}
&\hspace{-1 cm}\parallel T^{j,1}_{N} f \parallel^{2}_{L^{2}([0,1])}
\,=
\int_{0}^{1}\,|T^{j,1}_{N} f(s)|^{2}\,ds
\,\leq\,\int_{\mathbb{R}}\psi(s)\,|T^{j,1}_{N} f(s)|^{2}\,ds\\
&\hspace{-1 cm}=\;\int_{\mathbb{R}}\psi(s)\, T^{j,1}_{N} f(s)\;\overline{T^{j,1}_{N} f(s)}\,ds\,
=\;
\int_{\mathbb{R}}\psi(s)\,  \int_{{B}}\,\int_{{B}}\,e^{\imath N \left({|x|}^{j}-{|y|}^{j}\right)s}\,
f(x)\,\overline{f(y)}\,dx\,dy\,ds.
\end{align*}
Let $f\in L^{2}(B)$. Applying Fubini's theorem we get
\begin{align}\label{energy01}
\parallel T^{j,1}_{N} f \parallel^{2}_{L^{2}([0,1])}\;\leq\;
 \int_{{B}}\,\int_{{B}}\,K_{N}^{j,1}(x,y)\,
f(x)\,\overline{f(y)}\,dx\,dy.
\end{align}
In the light of the estimate (\ref{kernelsk1})
of Lemma \ref{kernelsk}, it follows that
\begin{align}\label{energy11}
\parallel T^{j,1}_{N} f \parallel^{2}_{L^{2}([0,1])}\;\lesssim\;
 \int_{{B}}\,\int_{{B}}\,
\frac{|f(x)|\,|f(y)|}{1+N^{2} \, \left({|x|}^{j}-{|y|}^{j}\right)^{2}}
\,dx\,dy.
\end{align}
Since $\displaystyle
\int_{0}^{1}\frac{dz}{1+N^2 z^2}\approx \frac{1}{N},\,$
then, applying Lemma \ref{kernelsq} with
$\,h(z)=\left(1+N^2 z^2\right)^{-1}\,$ to
the\\\\ estimate  (\ref{energy11}), we obtain
\begin{align}\label{e1}
\parallel T^{j,1}_{N} f \parallel_{L^{2}([0,1])}\;\lesssim\;
\frac{1}{\sqrt{N}}\,\parallel  f \parallel_{L^{2}(B)},
\qquad\text{for all dimensions}\;\;n\geq j.
\end{align}
To finish this case, it remains
to estimate $\,T^{2,1}f\,$ in the dimension
$\,n=1.$ In view of  (\ref{kernelsk1}) and (\ref{energy01}), we have
\begin{align*}
\hspace{-1 cm}
\parallel T^{2,1}_{N} f \parallel^{2}_{L^{2}([0,1])}\:\lesssim\,
 \int_{-1}^{1}\,\int_{-1}^{1}\,
\frac{|f(x)|\,|f(y)|}{1+N\,
\left|x^2-y^2\right|}
\,dx\,dy.
\end{align*}
Hence, by (\ref{even1}) of Lemma \ref{even},
\begin{align}\label{e2}
\parallel T^{2,1}_{N} f \parallel_{L^{2}([0,1])}\:
\lesssim\,\frac{1}{N^{1/4}}\,
\parallel f \parallel_{L^{2}([-1,1])}.
\end{align}
\textbf{The phase is quadratic in $\textbf{s}\,$ $\,(k=2)$}:\\
For $f\in L^{2}(B)$, using Fubini's theorem
then employing the estimate
(\ref{kernelsk2}) implies
\begin{align}\label{energy21}
\hspace{-0.6 cm}
\parallel T^{j,2}_{N} f \parallel^{2}_{L^{2}([0,1])}\:=\,
 \int_{{B}}\,\int_{{B}}\,K_{N}^{j,2}(x,y)\,
f(x)\,\overline{f(y)}\,dx\,dy
\,\lesssim\, G^{j}_{N}(f)+H^{j}_{N}(f)
\end{align}
where
\begin{align*}
G^{j}_{N}(f)\,=&\, \int_{{B}}\,\int_{{B}}\,
\frac{|f(x)|\,|f(y)|}{1+\sqrt{N}\, \sqrt{\left|{|x|}^{j}-{|y|}^{j}\right|}}
\,dx\,dy,\\
H^{j}_{N}(f)\,=&\, \int_{{B}}\,\int_{{B}}\,
\frac{|f(x)|\,|f(y)|}{1+N\, \left|{|x|}^{j}-{|y|}^{j}\right|}
\,dx\,dy.
\end{align*}
Since $\;\displaystyle  \int_{0}^{1}\,\frac{dz}{1+\sqrt{N}\,\sqrt{z}}
\,\approx\, \frac{1}{\sqrt{N}},\quad
\int_{0}^{1}\,\frac{dz}{1+N\,z}
\,=\,
\text{\large o}\left(\frac{1}{\sqrt{N}}\right),
\quad \text{as}\;\;\; N\longrightarrow+\infty,
$\\\\
then applying Lemma \ref{kernelsq}
to both $\,G^{j}_{N}(f)\,$ and $\,H^{j}_{N}(f)\,$ gives the estimate
\begin{align}\label{energy22}
G^{j}_{N}(f)+
H^{j}_{N}(f)\;\lesssim\; \frac{1}{\sqrt{N}}
\parallel f \parallel^{2}_{L^{2}(B)},
\qquad n\geq j.
\end{align}
It remains to control $\:G^{2}_{N}(f)\,$ and $\,
H^{2}_{N}(f)\:$ in the dimension $\,n=1.\,$
But when $\,n=1,$
\begin{align*}
\hspace*{-0.4 cm}
G^{2}_{N}(f)\,=&\, \int_{-1}^{1}\,\int_{-1}^{1}\,
\frac{|f(x)|\,|f(y)|}{1+\sqrt{N}\, \sqrt{\left|{x}^{2}-{y}^{2}\right|}}
\,dx\,dy\\ \leq&\,\frac{1}{\sqrt{N}}\,
\int_{-1}^{1}\,\int_{-1}^{1}\,
\frac{|f(x)|\,|f(y)|}{
\sqrt{\left|{x}^{2}-{y}^{2}\right|}}
\,dx\,dy
\,\lesssim\,
\frac{1}{\sqrt{N}}\,\parallel f \parallel^{2}_{L^{2}([-1,1])}\quad
\text{by}\;\; (\ref{even2})\; \text{of}\;
\text{Lemma} \;\ref{even}.
\end{align*}
An identical estimate holds for $H^{2}_{N}(f)$ in the dimension $n=1$ because of (\ref{even1}).
Combining this with (\ref{energy22})
and using them in (\ref{energy21}) yields
\begin{align}\label{e3}
\parallel T^{j,2}_{N} f \parallel_{L^{2}([0,1])}
\;\lesssim\;\frac{1}{{N}^{1/4}}
\parallel f \parallel_{L^{2}(B)}.
\end{align}
Finally, bringing the estimates (\ref{e1}),  (\ref{e2}) and (\ref{e3}) together results in (\ref{energy}).
\section*{References}
\bibliography{mybibfile}

\bigskip
\bigskip
Mathematics Department, Faculty of Science\\  Assiut University, Assiut,71516, Egypt\\
ahmed.abdelhakim@aun.edu.eg
\end{document}